\documentclass[12pt]{amsart}

\usepackage{graphicx}
\usepackage{amssymb}
\usepackage{float}
\usepackage{qtree}

\newtheorem{theorem}{Theorem}[section]
\newtheorem{lemma}[theorem]{Lemma}
\newtheorem{corollary}[theorem]{Corollary}
\newtheorem{prop}[theorem]{Proposition}

\theoremstyle{definition}

\theoremstyle{remark}
\newtheorem{remark}[theorem]{Remark}

\numberwithin{equation}{section}

\begin{document}

\title[Exceptions to the Erd\H os--Straus--Schinzel conjecture]{Exceptions to the\\ Erd\H os--Straus--Schinzel conjecture}
\author{Carl Pomerance}
\address{Department of Mathematics, Dartmouth College\\
Hanover, NH 03755}
\email{carlp@math.dartmouth.edu}
\author{Andreas Weingartner}
\address{Department of Mathematics, Southern Utah University\\
351 University Boulevard\\
Cedar City, UT 84720}
\email{weingartner@suu.edu}
\maketitle

{\it\centerline{For Krishnaswami Alladi on his 70th birthday}}

\begin{abstract}
A famous conjecture of Erd\H os and Straus is that for every integer $n\ge2$,
$4/n$ can be represented as $1/x+1/y+1/z$, where $x,y,z$ are positive integers.
This conjecture was generalized to $5/n$ by Sierpi\'nski, and then Schinzel
conjectured that for every integer $m\ge4$ there is a bound $n_m$ such that
the fraction $m/n$ is the sum of 3 unit fractions for all integers $n\ge n_m$.
Leveraging and generalizing work of Elsholtz and Tao, we show that if $n_m$ exists
it must be at least $\exp(m^{1/3+o(1)})$; that is, there are numbers $n$ this large
for which $m/n$ is not the sum of 3 unit fractions.  We prove a weaker, but numerically explicit
version of this theorem, showing that for $m\ge 6.52\times10^9$ there is a prime 
$p\in(m^2,2m^2)$ with $m/p$ not the sum of 3 unit fractions, and report on some
extensive numerical calculations that support this assertion with the much smaller bound $m\ge20$.  A result of Vaughan is  that for each $m$,
most $n$'s have $m/n$ representable; we make the dependence on $m$ in this result
explicit.  In addition,  we prove a result generalizing the problem
to the sum of $j$ unit fractions.
\end{abstract}

\section{Introduction}
Egyptian fractions have a long
and colorful history.  According to the Rhind Papyrus (ca.\ 1550 BCE), ancient Egyptians preferred
to write fractions as sums of unit fractions (fractions with numerator 1).  We have
not seen a compelling argument for {\it why} they had this preference, but nevertheless
it opened the door to many intriguing problems.   For surveys of some of the many problems and results,
see \cite[Ch. 4]{EG}, \cite[Sec. D11]{Gu}.

The Erd\H os--Straus conjecture, which dates to around 1948, asserts that $4/n$ is the sum of 3 unit fractions
for every integer
$n\ge 2$.  An early result of Obl\'ath \cite{O} is 
that $n$ has this property if $n+1$ is divisible by a prime $p\equiv3\pmod 4$.  This
implies that asymptotically all $n$ have the Erd\H os--Straus property.  In fact, exceptions
should only be divisible by primes $p$ for which $p+1$ has no prime factor $q\equiv3\pmod4$.
The number of such primes $p\le N$ is $O(N/\log^{3/2}N)$ as can be seen from sieve methods.
This property allows only $O(N/\log^{3/2} N)$ exceptional numbers up to $N$ via an argument akin to
the Hardy--Ramanujan inequality (cf.~Gottschlich \cite[Lemma 2.3]{G}).  The count of possible
exceptions has been strongly improved, though not recently: In 1970, Vaughan \cite{V} gave
 the upper bound $N/\exp(c\log^{2/3}\kern-2pt N)$ for a positive constant $c$.
 
 Sierpi\'nski conjectured that not only $4/n$, but also $5/n$, can be written as a sum
 of 3 unit fractions, and then Schinzel conjectured that for each integer $m\ge4$,
 $m/n$ is the sum of 3 unit fractions for all sufficiently large $n$, depending on the
 choice of $m$.  Clearly a necessary condition for ``sufficiently large" is that $n\ge m/3$.
 In this paper we show that ``sufficiently large" is indeed big, in fact larger than any
 fixed power of~$m$.
 
 \begin{theorem}
 \label{th-mbound}
 For each $\epsilon>0$ there is a bound $m(\epsilon)$ such that  for each $m\ge m(\epsilon)$
 there is some $n>\exp(m^{1/3-\epsilon})$ with $m/n$ not the sum of $3$ unit fractions.
 \end{theorem}
 For the proof we leverage some of the tools in Elsholtz--Tao \cite{ET}, which paper was
 principally concerned with the number of representations of $4/n$ as a sum of 3 unit
 fractions.  We also prove a version of Theorem \ref{th-mbound} that's weaker, but
 numerically explicit, and in particular we obtain the following result.
 \begin{theorem}
\label{th-explicit}
For each integer $m\ge 6.52\times10^9$ there is a prime $p\in(m^2,2m^2)$ for which
$m/p$ is not the sum of $3$ unit fractions.
\end{theorem}
  
Complementing our lower bounds we prove the following upper bound for the 
distribution of exceptions to the Erd\H os--Straus--Schinzel conjecture.
\begin{theorem}
\label{th-v}
There is an absolute positive constant $C$ such that for each pair $m,N$ with 
$4\le m \le \log^2 N$ the number of $n\le N$ with $m/n$ not the sum of $3$
unit fractions is at most $N/\exp(C\log^{2/3}(N)/\varphi(m)^{1/3})$.
\end{theorem}
Exploiting the large sieve, the proof is largely derivative of Vaughan's theorem in \cite{V}.

In our proof of Theorem \ref{th-mbound} we actually show that not only is there one
exceptional $n> \exp(m^{1/3-\epsilon})$, but that most prime values of $n$ near this
bound are exceptions.  This might be contrasted with Theorem \ref{th-v} which implies
that when $n\approx\exp(m^{1/2})$, most values of $n$ and in fact most primes
are not exceptions.  So between $\exp(m^{1/3})$ and $\exp(m^{1/2})$ there is
a transition from ``usually false" to ``usually true".  

The proof of Theorem \ref{th-mbound} suggests that the average number of solutions
for a prime $n=p \ge m$ is $\frac{\log^3 p}{m} (\log\log p)^{O(1)}$. 
If we ignore the $\log\log p $ factor and, as in \cite[Remark 1.1]{ET} with the case $m=4$, 
model the number of solutions at each prime $p$ as a Poisson process with intensity 
$\frac{\log^3 p}{m}$, we would expect any given prime $p$ to have ``probability"
$\exp(-(\log p)^3/m)$ of having no solution. 
This would suggest that most primes $p>\exp(m^{1/3+\epsilon})$ have solutions.
It also indicates that there are many exceptional primes 
$p>\exp(m^{1/2-\epsilon}) $, 
while there are no exceptional primes $p>\exp(m^{1/2+\epsilon})$, when $m$ is sufficiently large.

 We also consider the more general question of whether $m/n$ can be represented
 as the sum of $j$ unit fractions, showing that there are somewhat large
 exceptional values here as well.
 \begin{theorem}
 \label{th-mjbound}
 For each pair of positive integers $j,k$,  there is a number $m(j,k)$ such that for
 each $m\ge m(j,k)$, we have $m/(km+1)$ not the sum of $j$ unit fractions.
 \end{theorem}
 
 In addition, we examine the Erd\H os--Straus--Schinzel conjecture numerically for
 various values of $m$.  
  
 \section{Some basic thoughts}
 For $m,n\in \mathbb{N}$, we consider the equation 
\begin{equation}
\label{ESSn}
\frac m n=\frac1x+\frac1y+\frac1z.
\end{equation}
We say that a solution $(x,y,z) \in \mathbb{N}^3$ of \eqref{ESSn} is of Type I if $n$ divides $x$ but is coprime to $y,z$, and 
of Type II if $n$ divides $y,z$ but is coprime to $x$.  
Note that if $n$ is prime, $n \nmid m$ and $m\ge 4$ then,
up to permuting $x,y,z$, every
solution to \eqref{ESSn} must be of Type I or Type II; this is not necessarily the case when
$n$ is composite.  For example, $5/6=1/3+1/4+1/4$ is not of either type.

The following parametrizations of Type I and Type II solutions, as well as their proofs,
follow the ideas in \cite[Section 2]{ET}, where the case $m=4$ is treated.
Also, see Aigner \cite{A} and Nakayama \cite{N}.


\begin{prop}\label{proptype1par}
Let $n,m \in \mathbb{N}$. There exists a Type I solution $(x,y,z)\in \mathbb{N}^3$ of \eqref{ESSn} if and only if 
there exist $a, d, f \in \mathbb{N}$ with $f\mid ma^2d+1$, $mad\mid n+f$, and $(n+f)/mad$ coprime to $n$. 
\end{prop}

\begin{proof}
First assume that there exist $a, d, f \in \mathbb{N}$ with $f\mid ma^2d+1$, $mad\mid n+f$, and $(n+f)/mad$ coprime to $n$. 
Define $e:=\frac{ma^2d+1}{f}\in \mathbb{N}$, $c:=\frac{n+f}{mad}\in \mathbb{N}$, so that $c$ is coprime to $n$,
and $b:=ce-a\in \mathbb{N}$, since $ce>a$. Then one easily verifies that $(x,y,z):=(abdn,acd, bcd)$ satisfies \eqref{ESSn},
and that $mabd=ne+1$, which implies $\gcd(n,abd)=1$. Since $c$ is also coprime to $n$,
so are $y,z$, and the solution $(x,y,z)$ is of Type I.

Conversely, assume that  $(x,y,z)\in \mathbb{N}^3$ is a Type I solution of \eqref{ESSn}.
We factor $x = ndx'$, $y = dy'$, $z = dz'$, where $\gcd(x',y',z')=1$. After multiplying \eqref{ESSn}
by $ndx'y'z'$, we get
\begin{equation}\label{eqt11}
mdx'y'z' = y'z' + nx'y' + nx'z'.
\end{equation}
As $y',z'$ are coprime to $n$, we conclude that 
\begin{equation}\label{eqt12}
x'\mid y'z', \quad y'\mid x'z', \quad z'\mid x'y'.
\end{equation}
We claim that this  implies
\begin{equation}\label{eqt13}
x'=ab, \ y'=ac, \ z'=bc,
\end{equation}
where
$$
a=\gcd(x',y'),\quad b=\gcd(x',z'), \quad c=\gcd(y',z').
$$
Indeed, if a prime $p$ divides $x'y'z'$, then $\gcd(x',y',z')=1$ implies that (at least) one of $x',y',z'$ is 
not divisible by $p$, while \eqref{eqt12} implies that the other two, and hence their $\gcd$, are divisible 
by the same power of $p$. 
Substituting \eqref{eqt13} into \eqref{eqt11}, we obtain
\begin{equation}\label{eqt14}
mabcd = n(a+b)+c.
\end{equation}
As $y,z$ are coprime to $n$, $abcd$ is coprime to $n$ and
\eqref{eqt14} shows that $c\mid a+b$. Writing $e:=\frac{a+b}{c}\in \mathbb{N}$ and dividing \eqref{eqt14} by $c$, we have 
$mabd = ne+1$. Define $f:=macd-n$, so that $mad\mid n+f$. Since $(n+f)/mad=c$ and $c$ is coprime to $n$, so is $(n+f)/mad$.
We have $f\mid ma^2d+1$, as
$$
ef = emacd-en=emacd-(mabd-1)=mad(ec-b)+1=ma^2d+1.
$$
\end{proof}

The condition that $(n+f)/mad$ is coprime to $n$ is not necessary when $m\ge 4$ and $n$ is prime:

\begin{corollary}
\label{corI}
Let $m\ge 4$ and $p$ be prime. 
There exists a Type I solution $(x,y,z)\in \mathbb{N}^3$ of \eqref{ESSn} with $n=p$ if and only if 
there exist $a, d, f \in \mathbb{N}$ with $f\mid ma^2d+1$ and $mad\mid p+f$.
\end{corollary}

\begin{proof}
Assuming there exist $a, d, f \in \mathbb{N}$ with $f\mid ma^2d+1$ and $mad\mid n+f$, the solution $(x,y,z)$ is 
constructed as in the proof of Proposition \ref{proptype1par}, and we find again that $\gcd(n,abd)=1$.
Since $n$ is prime, if $n$ is not coprime to $c$ then $n\mid c$, hence $n\mid y$ and $n\mid z$, and
$1/x+1/y+1/z\le 3/n<m/n$. Thus $n$ must be coprime to $c$, $\gcd(n,abcd)=1$, and $(x,y,z)$ is of Type I.

The converse follows from Proposition \ref{proptype1par}.
\end{proof}


\begin{prop}\label{proptype2par}
Let $n,m \in \mathbb{N}$. There exists a Type II solution $(x,y,z)\in \mathbb{N}^3$ of \eqref{ESSn} if and only if 
there exist $a,b,e \in \mathbb{N}$ with $e\mid a+b$, $mab\mid n+e$, and $(n+e)/m$ coprime to $n$. 
\end{prop}

\begin{proof}
First assume that there exist $a,b,e \in \mathbb{N}$ with $e\mid a+b$ and $mab\mid n+e$, and$(n+e)/m$ coprime to $n$. 
Define $c:=\frac{a+b}{e}\in \mathbb{N}$ and $d:=\frac{n+e}{mab}\in \mathbb{N}$.
Then one easily verifies that $(x,y,z):=(abd,acdn,bcdn)$ satisfies \eqref{ESSn}. 
Since $x:=abd=(n+e)/m$, $x$ is coprime to $n$ and $(x,y,z)$ is a Type II solution.

Conversely, assume that  $(x,y,z)\in \mathbb{N}^3$ is a Type II solution of \eqref{ESSn}.
We factor $x = dx'$, $y = ndy'$, $z = ndz'$, where $\gcd(x',y',z')=1$. After multiplying \eqref{ESSn}
by $ndx'y'z'$, we get
\begin{equation}\label{eqt21}
mdx'y'z' = ny'z' + x'y' + x'z'.
\end{equation}
As $x'$ is coprime to $n$, we conclude that 
$x'\mid y'z'$, $y'\mid x'z'$, $z'\mid x'y'$.
As in the proof of Proposition \ref{proptype1par}, this implies \eqref{eqt13}.
Substituting \eqref{eqt13} into \eqref{eqt21}, we obtain
\begin{equation}\label{eqt24}
mabcd = a+b+nc,
\end{equation}
which shows that $c\mid a+b$. Define $e:=\frac{a+b}{c}\in \mathbb{N}$, so that $e\mid a+b$. Dividing \eqref{eqt24} by $c$, we have 
$mabd = e+n,$
that is $mab\mid e+n$. Since $\frac{e+n}{m}=abd=x$ and $x$ is coprime to $n$, so is $ \frac{e+n}{m}$.
\end{proof}

The condition that $(n+e)/m$ is coprime to $n$ is not necessary when $m\ge 4$ and $n$ is prime:

\begin{corollary}
\label{corII}
Let $m\ge 4$ and $n$ be prime. 
There exists a Type II solution $(x,y,z)\in \mathbb{N}^3$ of \eqref{ESSn} if and only if 
there exist $a,b,e \in \mathbb{N}$ with $e\mid a+b$ and $mab\mid n+e$.
\end{corollary}

\begin{proof}
Assuming there exist $a,b,e \in \mathbb{N}$ with $e\mid a+b$ and $mab\mid n+e$, the solution $(x,y,z)$ is 
constructed as in the proof of Proposition \ref{proptype2par}. Since $n$ is prime, if $n$ is not coprime to $x$ then $n\mid x$ and
$1/x+1/y+1/z\le 3/n<m/n$. Thus $n$ must be coprime to $x$ and $(x,y,z)$ is of Type II.

The converse follows from Proposition \ref{proptype2par}.
\end{proof}

\section{A lower bound for the number of exceptional primes}

We will make use of the Brun--Titchmarsh inequality, which states that the number of primes up to $N$ that are congruent to $a \pmod q$ is
$$
\pi(N;q,a) \ll \frac{N}{\varphi(q) \log (N/q)} \qquad (q<N),
$$
where $\varphi(q)$ is the Euler totient function.   We record \eqref{ESSn} in the case that
$n=p$, a prime:
\begin{equation}
\label{ESSp}
\frac mp = \frac1x+\frac1y+\frac1z.
\end{equation}

\begin{theorem}
\label{th-mboundtoo}
There is a constant $c>0$, such that for every integer $m\ge 8$, there are more than 
$$
\exp\{c\, \varphi(m)^{1/3}/(\log m)^{2/3}\}
$$ 
primes $p$  
for which \eqref{ESSp} has no solution in natural numbers $x,y,z$.
\end{theorem}

The proof follows the ideas in \cite[Sections 8, 9]{ET}, generalizing from $m=4$ to general $m$. 

\begin{proof}
Note that \eqref{ESSp} cannot be solved when $p=2,3$ for $m=8$ and $m\ge10$,
nor when $p=2,5$ for $m=9$, so
 we may assume that $m$ is large. 
When $p$ is prime, then all solutions to \eqref{ESSp} are of Type I or Type II, as discussed
above.  If $(x,y,z)$ is a solution to \eqref{ESSp} of Type I, Corollary \ref{corI}
shows that there are natural numbers $a,d,f$ such that
$$
p\equiv -f \pmod{mad}, \qquad f\mid ma^2 d+1.
$$
By Lemma \ref{madub}, the modulus satisfies $mad \le 3p \le 3N$, provided $p\le N$.
For given $m,a,d,f$, the number of primes $p\le N$ satisfying $p\equiv - f \pmod{ mad}$ is 
$$ \ll \frac{N}{\varphi(mad) \log(2+ \frac{N}{mad})} \le \frac{N}{\varphi(m) \varphi(ad) \log(2+ \frac{N}{mad})} .$$
This follows from Brun--Titchmarsh when $mad \le N/2$, while the count is clearly $O(1)$ if $N/2<mad \le 3N$. 
The number of $p\le N$ covered by these congruences, by varying the parameters $a, d, f$, is 
\begin{equation}\label{type1ub}
\ll   \frac{N}{\varphi(m)} \sum_{a,d: ad \le 3N/m}\frac{\tau(ma^2 d +1)}{ \varphi(ad) \log(2+ \frac{N}{mad})} .
\end{equation}
The analogue of the estimate \cite[Eq. (8.2)]{ET} is  
\begin{equation}\label{eq82}
\sum_{a,d: X/2 \le ad \le X}\frac{\tau(ma^2 d +1)}{ \varphi(ad) } \ll \log^2 X \log m \qquad (X,m \ge 2,\ m\ll X^{O(1)}),
\end{equation}
which is proved just like in \cite{ET} with $m$ replacing $4$.
Splitting the sum in \eqref{type1ub} into dyadic intervals $\frac{3N}{m}2^{-j-1}\le ad \le \frac{3N}{m}2^{-j}$, 
the contribution to \eqref{type1ub} from $j$ with $m^{1/6} \le \frac{3N}{m}2^{-j}$ is
$$
\ll \frac{N}{\varphi(m)} \log^2 N \log m \sum_{ j \ll \log N} \frac{1}{j} \ll \frac{N}{\varphi(m)} \log^2 N \log\log N \log m ,
$$
by \eqref{eq82}.
Since $\tau(ma^2d+1)\ll (mad)^{1/6}$, 
the contribution to \eqref{type1ub} from $j$ with $m^{1/6} > \frac{3N}{m}2^{-j}$, that is $ad < m^{1/6}$, is
$$
\ll \frac{N}{\varphi(m)}  \sum_{a,d: ad \le m^{1/6} }m^{1/6} \ll \frac{N}{\varphi(m)}  m^{1/2}.
$$
Assuming that $N$ is chosen so that $e^{m^{1/4}} \ll N<e^m$, the expression in \eqref{type1ub}, 
and hence the number of primes  $p\le N$ covered by Type I solutions, is
$$
\ll \frac{N}{\varphi(m)}\log^2 N \log^2 m.
$$
We now specify $N=N(m)$ as  the solution to
$$\varphi(m)/\log^2 m= C \log^3 N ,$$ 
for some sufficiently large constant $C$, noting that this is consistent with $e^{m^{1/4}} \ll N<e^m$ when $C\ge1$.
Then most primes $p\in (N/2,N]$
are not covered by these congruences, and thus have no Type I solution to \eqref{ESSp}.

It remains to count the number of primes $p\le N$ covered by Type II solutions.
If $(x,y,z)$ is a solution to \eqref{ESSp} of Type II, Corollary \ref{corII} shows that there are natural numbers $a,b,e$ with
$$
p\equiv - e \pmod{mab}, \qquad e\mid a+b.
$$
Note that $\gcd(a,b)=1$ follows from $\gcd(x',y',z')=1$ in the proof of Proposition \ref{proptype2par}.
Since $e\mid a+b$, we have $e\le a+b \le 2ab$. And $mab\mid p+e$ implies $mab \le p+e \le p + 2ab$, so $(m-2) ab \le p$ and 
\begin{equation}\label{type2mabub}
mab \le p \frac{m}{m-2} \le 2p \le 2N.
\end{equation}
For given $m,a,b,e$, the number of primes $p\le N$ satisfying $p\equiv - e \pmod{mab}$ is 
$$
\ll \frac{N}{\varphi(mab) \log(2+\frac{N}{mab})}\le \frac{N}{\varphi(m) \varphi(ab) \log(2+\frac{N}{mab})},
$$
again by Brun--Titchmarsh if $mab\le N/2$ and trivially if $N/2 < mab \le 2N$. 
The number of primes $p\le N$ that can be covered by these congruences, by varying the parameters $a, b, e$, is
$$
\ll \frac{N}{\varphi(m)} \sum_{a,b: ab \le 2N/m \atop (a,b)=1} \frac{\tau(a+b)}{\varphi(ab)   \log(2+\frac{N}{mab})}.
$$
Splitting this sum into dyadic intervals $2^{j-1}< \frac{N}{mab} \le 2^{j}$,
and estimating the resulting sums as in the last paragraph of \cite[Section 9]{ET}, we find that  
the number of primes $p\le N$ covered by these congruences is 
$$
\ll \frac{N}{\varphi(m)} \log^2 N \log\log N.
$$
Thus, most primes $p\in (N/2,N]$ are covered neither by Type I nor by Type II congruences if $\varphi(m)/\log^2 m= C \log^3 N$ and $C$ is large enough, that is
$$
N = \exp\{ (\varphi(m)/C\log^2 m)^{1/3}\}.
$$ 
The result now follows with $c=1/(2C^{1/3})$,
since we have $\frac{N}{4\log N} >\sqrt{N}= \exp\{c (\varphi(m)/\log^2 m)^{1/3}\}$.
\end{proof}

To see Theorem \ref{th-mbound},  since $\varphi(m)\gg m/\log\log m$, it follows 
from Theorem \ref{th-mboundtoo} that for each $\epsilon>0$
and $m\ge m(\epsilon)$, we have more than $\exp(m^{1/3-\epsilon})$ primes $p$
where \eqref{ESSp} has no solution.

\section{An upper bound}
In this section we prove Theorem \ref{th-v}.
Our proof largely follows the argument in Vaughan \cite{V}.

For $m\ge4$ and a prime $p\equiv-1\pmod m$, let $f(p)=f_m(p)$ denote the greatest
integer that is at most
\begin{equation}
\label{eq:vsum}
\frac12 \sum_{t\,|\,(p+1)/m}|\mu(t)|\tau\Big(\frac{p+1}{tm}\Big).
\end{equation}
For other primes $p$ we let $f(p)=0$.  As shown in \cite{V} for each $p$ there
are at least $f(p)$ residue classes mod $p$ such that if $n$ lies in one of them, then
$m/n$ is a sum of 3 unit fractions.  

The strategy is to use the large sieve to
show that the number of $n\le N$ lying outside of these $f(p)$ residue classes
mod $p$ for each $p$ is bounded above by the bound in Theorem \ref{th-v}.
To achieve this, we first establish the following lemma.

\begin{lemma}
\label{lem-sum}
For $m\le(\log x)^{O(1)}$ we have
\[
\sum_{p\le x}\frac{f(p)}{p}\asymp \frac1{\varphi(m)}(\log x)^2.
\]
\end{lemma}
\begin{proof}
Via partial summation it suffices to show that
\begin{equation}
\label{eq-ps}
\sum_{p\le x}f(p)\asymp \frac1{\varphi(m)}x\log x.
\end{equation}
For the upper bound we use the simple inequality
\[
\tau_3(n)\le3\sum_{\substack{d\,|\,n\\d\le n^{2/3}}}\tau(d),
\]
where $\tau_3(n)$ is the number of triples $a,b,c$ of integers with $abc=n$
(cf. Koukoulopoulos \cite[Ex. 20.2]{K}).  Then for a prime $p\equiv-1\pmod m$,
\[
f(p)< \tau_3((p+1)/m)\le 3\sum_{\substack{d\,|\,(p+1)/m\\d\le p^{2/3}}}\tau(d).
\]
Thus, via the Brun--Titchmarsh inequality and our upper  bound on $m$,
\begin{align*}
\sum_{p\le x}f(p)&\ll\sum_{d\le x^{2/3}}\tau(d)\pi(x;dm,-1)\ll\sum_{d\le x^{2/3}}\frac {x\tau(d)}{\varphi(dm)\log x}\\
&\le\frac x{\varphi(m)\log x}\sum_{d\le x^{2/3}}\frac{\tau(d)}{\varphi(d)}\ll \frac1{\varphi(m)}x\log x.
\end{align*}

For the lower bound first note that for an integer $n\ge2$, one has $\lfloor n/2\rfloor\ge n/3$.
If $p\equiv-1\pmod m$ and $p+1>m$, then the sum in
\eqref{eq:vsum} is $\ge2$, so that
\begin{align*}
\sum_{p\le x}f(p)&\ge\frac13\sum_{\substack{p\le x\\p\,\equiv\,-1\pmod m\\p+1>m}}\sum_{dt\,|\,(p+1)/m}|\mu(t)|\\
&\ge
\frac13\sum_{\substack{d,t\le x^{1/6}\\dt>1}}|\mu(t)|\sum_{\substack{p\le x\\p\,\equiv\,-1\kern-5pt\pmod{mdt}}}1\\
&\ge \frac13\sum_{\substack{d,t\le x^{1/6}\\1\le\Omega(dt)\le3\log\log x}}|\mu(t)|\sum_{\substack{p\le x\\p\,\equiv\,-1\kern-5pt\pmod{mdt}}}1,
\end{align*}
where $\Omega(n)$ is the total number of prime factors of $n$ with multiplicity.
We now use the Bombieri--Vinogradov theorem noting that the number of triples $m,d,t$
with given product $q$ is at most $(\log x)^{O(1)}$.  Thus, 
\[
\sum_{p\le x}f(p)\gg
\sum_{\substack{d,t\le x^{1/6}\\1\le \Omega(dt)\le3\log\log x}}\frac{|\mu(t)|x}{\varphi(mdt)\log x}
\gg \frac1{\varphi(m)}x\log x,
\]
completing the proof of \eqref{eq-ps} and the lemma.
\end{proof}

 Let $N$ be large, let $X\le N^{1/2}$ be a quantity specified later, and let $P=\prod_{p\le X}p$.
We now employ the large sieve.  Let
\[
S=\sum_{\substack{s\le N^{1/2}\\s\,|\,P}}|\mu(s)|\prod_{p\,|\,s}\frac{f(p)}{p-f(p)}.
\]
The number of $n\le N$ that avoid the $f(p)$ residue classes mod $p$ for each $p\le X$
is bounded above by $4N/S$, so our task is to get a lower bound for $S$.

 Let 
\[
G=\sum_{s\,|\,P}\prod_{p\,|\,s}\frac{f(p)}{p-f(p)}.
\]
For any $v\ge0$,
\[
G-S=\sum_{\substack{s>N^{1/2}\\s\,|\,P}}\prod_{p\,|\,s}\frac{f(p)}{p-f(p)}\\
\le N^{-v/2}\sum_{s\,|\,P}s^v\prod_{p\,|\,s}\frac{f(p)}{p-f(p)}.
\]
Thus,
\begin{align*}
\frac{G-S}G&\le N^{-v/2}\prod_{p\le X}\Big(1+p^v\frac{f(p)}{p-f(p)}\Big)\Big(\frac{p-f(p)}{p}\Big)\\
&= N^{-v/2}\prod_{p\le X}\Big(1+\frac{(p^v-1)f(p)}{p}\Big).
\end{align*}
We choose $v=1/\log X$, so that
\begin{align*}
\frac{G-S}G&\le\exp\Big(-\frac{\log N}{2\log X}+\sum_{p\le X}\frac{(e-1)f(p)}p\Big)\\
&\le\exp\Big(-\frac{\log N}{2\log X}+\frac{(e-1)C_2}{\varphi(m)}\log^2X\Big),
\end{align*}
where $C_2$ is the upper bound constant  implied in Lemma \ref{lem-sum}.
Let
\[
A=\frac12\Big(\frac{\varphi(m)}{(e-1)C_2}\Big)^{1/3}
\]
and choose 
\[
X=\exp\big(A(\log N)^{1/3}\big).
\]
Then
\begin{align*}
\frac{G-S}{G}&\le\exp\Big(-\frac1{2A}(\log N)^{2/3}+\frac1{8A}(\log N)^{2/3}\Big)\\
&<\exp\Big(-\frac1{4A}(\log N)^{2/3}\Big).
\end{align*}
We may assume this last expression is $<1/2$, else the theorem holds trivially, so
 $S> G/2$.  But
\[
G\ge \exp\Big(\sum_{p\le X}\frac{f(p)}p\Big)\ge\exp\Big(\frac{C_1}{\varphi(m)}(\log X)^2\Big),
\]
where $C_1$ is the constant in the lower bound implicit in Lemma \ref{lem-sum}.
Putting in our choice for $X$ we have
\[
\frac{4N}{S}\ll \frac NG\le N/\exp\Big(\frac{C_1A^2}{\varphi(m)}(\log N)^{2/3}\Big).
\]
It remains to note that $A^2/\varphi(m)\asymp1/\varphi(m)^{1/3}$, completing
our argument for Theorem \ref{th-v}.

\section{The general case}
In this section we prove Theorem \ref{th-mjbound}.

Let $a_1>a_2>\dots$ be a sequence of real numbers with $\lim a_n=0$ and let
$\mathcal A=\{a_1,a_2,\dots\}$.
For each positive integer $j$, let $V_j$ denote the subset of $\mathcal A^j$ where
the coordinates form a monotone non-increasing sequence.  Further let 
$T_j$ be the subset of $(\mathcal A\cup\{0\})^j$ again with the coordinates non-increasing.
For $v\in T_j$, let $s(v)$ denote the sum of the coordinates of $v$, and let $S_j=s(V_j)$.

 \begin{lemma}
 For $j\ge1$, the set of limit points of $V_j$ is $T_j\setminus V_j$.
 \end{lemma}
 \begin{proof}
 Suppose $(v_n)$ is an infinite sequence of distinct members of $V_j$ with $\lim v_n=w$.
 Let $v_n=(a_{n,1},\dots,a_{n,j})$.  The sequence $(a_{n,j})_n$ is
 either eventually constant or has limit 0.  The first option cannot occur since otherwise
 there are only finitely many choices for the vectors $v_n$.  Next, we consider $(a_{n,j-1})_n$
 and here both options are possible.  But if it is eventually constant, then all earlier
 coordinates of the vectors $v(s_n)$ likewise become eventually constant.  Continuing
 in this manner, we have that $v_n$ converges to a vector $w\in T_j$ with last coordinate 0,
 i.e., $w\in T_j\setminus V_j$.
  
 Conversely, if $t\in T_j$ with last coordinate 0, let $t=(t_1,\dots,t_k,0,\dots,0)$, where
 $t_1,\dots,t_k\in\mathcal A$ and $k<j$.  Suppose that $t_k=a_m$.   Replacing
 each of the 0's with $a_{m+n}$, we then have a sequence of vectors $t_n\in V_j$ 
 that converges to $t$.  This completes the proof.
 \end{proof}
 
\begin{lemma}
For each positive integer $j$ and each positive real $x$ there is a positive number $\epsilon$,
depending on the choice of $j$, $x$ and sequence $(a_n)$,
such that the interval $(x-\epsilon,x)$ contains no member of $S_j$.
\end{lemma}
\begin{proof}
For each fixed $j$ there is no infinite strictly increasing 
sequence made up of members of $S_j$.  To see this, we suppose such a sequence $(s_n)$
exists and let $s_n=s(v_n)$.   Write $v_n=(a_{n,1},\dots,a_{n,j})$.  Each of the sequences
$(a_{n,i})_n$ has 0 as a limit point or it repeats some nonzero number infinitely often, so
by passing to an infinite subsequence we may assume that either the sequence of $i$th
coordinates is constant or has limit 0, and this holds for each $i$.  These possibilities
are incompatible with $s(v_n)$ strictly increasing,
 which proves that no
infinite strictly increasing sequence can be formed from the elements of $S_j$.  Thus,
the assertion in the lemma holds.
\end{proof}

We now specify that the numbers $a_i$ are unit fractions.
To prove Theorem \ref{th-mjbound}, we use the lemmas with $a_i=1/i$, and note that there is some $\epsilon>0$,
depending on the choice of $k,j$, such that
$(1/k-\epsilon,1/k)$ contains no member of $S_j$.  However, for $m$ sufficiently large,
\[
m/(km+1)=1/(k+1/m)
\]
is in this interval, so it must be that that $m/(km+1)\notin S_j$.
This completes the proof of Theorem \ref{th-mjbound}.

\section{Empirical data}

The original Erd\H os--Straus conjecture was verified up to $10^{17}$ by 
Salez \cite{S}, and this was recently improved to $10^{18} $ by Mihnea--Dumitru \cite{MD}.
By sifting with the seven congruences in \cite[Proposition 1.9]{ET}, 
with 4 replaced by $m$,
we have verified the $m=5$ case up to $10^{18}$, the $m=6, 7, 8$ cases up to $10^{13}$,
and the $m=9,\dots,15$ cases up to $10^{12}$, with the noted exceptions found.
This sifting was done only with primes, and then composites made up of exceptional primes
were checked directly.
In all the cases any other exceptional $n$, if they exist, must exceed the stated $N$.
See Table 1.

In addition, with the help of a computer we verified that the claim in Theorem \ref{th-explicit} also holds for all $m \in [16,30000]$, except $m=19$.  Note too that from Table 1 we see that it holds for
$m=10$ and $m\in[12,15]$.  We conjecture that for every $m\ge20$ there is a prime
$p\in(m^2,2m^2)$ for which $m/p$ is not the sum of 3 unit fractions.

\begin{table}[h]\label{table1}
 \begin{tabular}{ | c | l | c | c | c |}
    \hline
$m $ & all exceptions $n\le N$ & Count & $N$ \\ \hline
$4$ & $1$ & 1 & $10^{18}$ \cite{MD} \\ \hline
$5$ & $1$ & 1 & $10^{18}$ \\ \hline
$6$ & $1$ & 1 & $10^{13}$ \\ \hline
$7$ & $1, 2$ & 2 & $10^{13}$ \\ \hline
$8$ & $1, 2, 3, 11, 17, 131, 241$ &  7  & $10^{13}$ \\ \hline
$9$ & $1, 2, 5, 11, 19$ &  5 & $10^{12}$ \\ \hline
$10$ & $1, 2, 3, 7, 11, 43, 61, 67, 181$ & 9 & $10^{12}$ \\ \hline
$11$ & $1, 2, 3, 4, 37$ &  5 & $10^{12}$ \\ \hline
 & $1, 2, 3, 5, 7, 13, 25, 29, 31, 37, 73, 97, 193, 433,$ & &\\ 
$12$& $ 577,1129,1657, 1873, 2521, 2593, 3433, 10369,$ & 24 & $10^{12}$ \\ 
& $12049, 12241$ & &\\ \hline
$13$ & $1, 2, 3, 4, 5, 7, 14, 53, 61, 67, 79, 211, 281$ &  13 & $ 10^{12}$ \\ \hline
$14$ & $1, 2, 3, 4, 5, 17, 19, 29, 59, 257, 353, 841$ &  12 & $10^{12}$\\ \hline
 & $1, 2, 3, 4, 8, 16, 17, 19, 23, 31, 34, 47, 53,61,79,$ &  &  \\ 
$15$& $113,122, 137,151,197,226,233,271,541,1103,$ & 32 & $10^{12}$ \\
& $1171,1367,4201,6301, 12601,16831,20521$ & &\\ \hline
     \end{tabular}
\medskip
\caption{Values of $n\le N$ for which \eqref{ESSp} has no solution. }
\end{table}

\section{Numerically explicit estimates: primes with Type I representations}
\label{secI}

Let $\tau(n)$ denote the number of divisors of $n$ and for $j\mid 6$,  let $\tau'_j(n)$ denote
the number of  divisors $d$ of $n$ with $\gcd(d,6)=j$.
In the proof, we will make use of the following estimates.
\begin{lemma}\label{lem1}
For all integers $n\ge2$ we have
\begin{align*}
\tau(n) &\le 138.32(n-1)^{1/6},\\
\tau'_1(n)& \le 16.2(n-1)^{1/6},\\
\tau'_2(n)&\le 51.3(n-1)^{1/6},\\
\tau'_3(n)&\le 32.3(n-1)^{1/6},\\
\tau'_6(n)&\le102.7(n-1)^{1/6}.
\end{align*}
\end{lemma}
\begin{proof}
First notice that $\tau(p^a)/p^{a/6}=(a+1)/p^{a/6}\le 1$ unless $p\le61$.  Further,
for  $p\le 61$ we compute the integer $a_p$ that maximizes $(a+1)/p^{a/6}$; these
maximizing prime powers being 
\[
2^8,3^4,5^3,7^2,11^2,13,17,19,23,29,31,37,41,43,47,53,59,61.
\]
Thus, if $u$ is the product of these prime powers, then 
\[
\tau(n)/n^{1/6}\le\tau(u)/u^{1/6}<138.313.
\]
To see the claimed inequality, note that $138.313n^{1/6}<138.32(n-1)^{1/6}$ for $n\ge3294$
and the assertion is easily checked for smaller values of $n$.  The inequalities for $\tau'_j$
are proved in a similar manner.
\end{proof}

We are aware that the exponent ``$1/6$" here can be replaced with any fixed $\epsilon>0$
at the expense of larger coefficients,
and there are even effective asymptotic estimates for the maximal order
(see \cite{NR}), but the elementary
Lemma \ref{lem1} is optimal for our needs.

\begin{lemma}
\label{divisors}
Let $j\,|\, 6$ and let $n$ be a positive integer.  
Among the divisors $d$ of $n$ with $\gcd(d,6)=j$ at most $\frac12\tau'_j(n)$
of them have $d>\sqrt{jn}$.
\end{lemma}
\begin{proof}
Write $n=2^i3^kv$ with $\gcd(v,6)=1$ and $i,k\ge0$.
In the case $j=1$ the divisors of $n$ coprime to 6
are precisely the divisors of $v$, and among these at most half of them are $>\sqrt{v}$.
But $\sqrt{n}\ge\sqrt{v}$, so the case $j=1$ is proved.  If $j=2$ note that the divisors $2d$ of $n$
coprime to 3, correspond to divisors $d$ of $n/2$ coprime to 3 with at most half of these
$>\sqrt{2^{i-1}v}$.  Note that $2d>\sqrt{2n}$ implies that $d>\sqrt{n/2}\ge\sqrt{2^{i-1}v}$.
The other cases are proved similarly.
\end{proof}

\begin{lemma}
\label{lem3}
For positive integers $r\le k$, and for $x\ge r$, we have
\begin{align*}
\sum_{\substack{n\le x\\n\,\equiv\,r\kern-5pt\pmod k}}n^{-\alpha}
&\le r^{-\alpha}+\frac1k(1-\alpha)^{-1}(x^{1-\alpha}-r^{1-\alpha}), \qquad (0\le\alpha<1),\\
\sum_{\substack{n\le x\\n\,\equiv\,r\kern-5pt\pmod k}}n^\alpha
&<\frac1{k(1+\alpha)}(x+k)^{1+\alpha}, \qquad(\alpha>0).
\end{align*}

\end{lemma}
\begin{proof}
These  inequalities are easy exercises.
\end{proof}

As discussed above, when $p$ is prime, $p\,\nmid\, m$ and $m\ge 4$, then all solutions to \eqref{ESSp} 
are of Type I or Type II.  By Proposition \ref{proptype1par},  if $(x,y,z)$ is a solution to \eqref{ESSp} of Type I, there are natural numbers $a,d,f$ such that
$$
p\equiv -f \pmod{mad}, \qquad f\mid ma^2 d+1.
$$
\begin{lemma}\label{madub}
When $p$ has a Type I solution to \eqref{ESSp} with $m\ge4$, then $mad \le 2p+1$.
\end{lemma}
\begin{proof}
As in the proof of Proposition \ref{proptype1par}, we define the natural numbers
$e:=\frac{ma^2d+1}{f}$, $c:=\frac{p+f}{mad}$, $b:=ce-a$,
where
$
(x,y,z)=(abdp,acd,bcd).
$
We may assume $y\le z$, i.e. $a\le b$. Then $ce=a+b\le 2b$ and 
$c \le \frac{2b}{e} = \frac{2}{ef} bf$. 
The definitions of $e,c,b$ imply that $bf=pa+c$, so
$$ bf =pa+c \le pa +  \frac{2}{ef} bf \le pa + \frac{2}{m+1} bf,$$
since $ef=ma^2d+1\ge m+1$. Thus,
$$
bf \le pa \frac{m+1}{m-1}
$$
and, since $b\ge a$, 
$$
f\le p \frac{m+1}{m-1}.
$$
Now
$$
macd = p+f \le p\frac{2m}{m-1}.
$$
If $c\ge 2$, we obtain $mad \le p \frac{m}{m-1} < 2p$. If $c=1$, then $bf=pa+c=pa+1$ and $f=pa/b+1/b \le p+1$, so 
$mad=p+f\le 2p+1$. 
\end{proof}

For given $m,a,d,f$, we wish to count  the number of primes $p\in(N/2, N]$ satisfying $p\equiv - f \pmod{mad}$.  We shall consider 4 cases depending on the value of $\gcd(f,6)$.

\subsection{The case $\gcd(f,6)=1$.}  In this case, for given values of $m,a,d$, we
let $f$ run over the divisors of $ma^2d+1$ coprime to 6.  The number of primes
$p\in(N/2,N]$ with $p\equiv-f\pmod{mad}$ is at most
$$ 
\le \left\lceil \frac{N}{2mad}\right\rceil\le\left\lfloor\frac{N}{2mad}\right\rfloor+1.
$$
Thus, by
 Lemma \ref{lem1}, with $\kappa_1=16.2$, the number of primes in this case is at most
 \begin{align*}
 \sum_{mad\le 2N+1}&\left(\left\lfloor\frac N{2mad}\right\rfloor+1\right)\tau'_1(ma^2d+1)\\
& \le \frac{\kappa_1N}{2m^{5/6}}\sum_{mad\le N/2}a^{-2/3}d^{-5/6}+\sum_{mad\le2N+1}\tau'_1(ma^2d+1)\\
 &=S_{1,1}+S_{1,2},
 \end{align*}
 say,  since $\lfloor N/2mad \rfloor $ vanishes unless $mad\le N/2$.
Thus, by Lemma \ref{lem3},
\[
S_{1,1}\le\frac{\kappa_1 N}{2m^{5/6}}\sum_{d\le N/2m}3\Big(\frac N{2md}\Big)^{1/3}d^{-5/6}
<\frac{3\kappa_1 N^{4/3}}{2^{4/3}m^{7/6}}\zeta(7/6)<127.1\frac{N^{4/3}}{m^{7/6}}.
\]

We will work harder for $S_{1,2}$.  We first consider 2 cases: $mad>1.01N$ and $mad\le1.01N$.
In the first case, since $madc-f\le N$, we have $f>N/100$.  But $mad\le2N+1$, so we have
$f>mad/300$, and so 
\begin{equation}
\label{eq:large}
f^2>\frac{m^2a^2d^2}{10^5}>10ma^2d>6(ma^2d+1),
\end{equation}
assuming $m\ge10^9$, say.  So, we are only considering divisors of $ma^2d+1$ larger than
the square root.  Thus, by Lemma \ref{divisors},
\begin{equation}
\label{eq:twosums}
S_{1,2}\le\frac12\sum_{mad\le 1.01N}\tau'_1(ma^2d+1)+\frac12\sum_{mad\le 2N+1}\tau'_1(ma^2d+1).
\end{equation}
The two sums are computed similarly; let $X$ stand for either $1.01N$ or $2N+1$.

Considering first the case when $a\le100$, we have
\begin{align*}
\frac12\sum_{a\le100}\sum_{d\le X/ma}\tau'_1(ma^2d+1)&\le
\frac12\kappa_1 m^{1/6}\sum_{a\le100}a^{1/3}\sum_{d\le X/ma}d^{1/6}\\
&\le \frac37 \kappa_1 m^{1/6}\sum_{a\le100} a^{1/3}\Big(\frac{X}{ma}+1\Big)^{7/6}\\
&<\frac37\kappa_1 m^{1/6}\sum_{a\le100}a^{1/3}\Big(\frac{1.01X}{ma}\Big)^{7/6}\\
&<\frac37\kappa_1 (1.01X)^{7/6}m^{-1}7.51.
\end{align*}
Here we directly computed the $a$-sum and assumed that $X>m^2$ and that $m\ge10^9$.  To this we will add the case $a>100$:
\begin{align*}
\frac12\sum_{a>100}\sum_{mad\le X}\tau'_1(ma^2d+1)&\le
\frac12\kappa_1 m^{1/6}\sum_{d<X/100m}d^{1/6}\sum_{100<a\le X/md}a^{1/3}\\
&\le\frac38\kappa_1 m^{1/6}\sum_{d<X/100m}d^{1/6}\Big(\frac{X}{md}+1\Big)^{4/3}\\
&<\frac38\kappa_1 (1.01X)^{4/3}m^{-7/6}\zeta(7/6).
\end{align*}
Adding these 2 estimates, each for $X=1.01N$ and $X=2N+1$, we get
\[
S_{1,2}<143.4\frac{N^{4/3}}{m^{7/6}}+171.8\frac{N^{7/6}}m.
\]
So, our estimate in the case that $f$ is coprime to 6 is
\begin{align*}
S_{1,1}+S_{1,2}&<270.5\frac{N^{4/3}}{m^{7/6}}+171.8\frac{N^{7/6}}m\\
&=\Big(270.5+\frac{171.8}{(N/m)^{1/6}}\Big)\frac{N^{4/3}}{m^{7/6}}
<276\frac{N^{4/3}}{m^{7/6}},
\end{align*}
assuming that $N>m^2$ and $m\ge10^9$.

\subsection{The case $\gcd(f,6)=2$.}
If $6\mid m$, then we must have $f$ coprime to 6 (else $madc-f$ is not prime),
so this last estimate stands for our bound for Type I solutions.
Otherwise we have more work to do.  In the current case $f$ is even, so
that we only consider values of $m,a,d,c$ that are all odd.  For given values of $m,a,d$,
the number of odd integers $c$ that place $madc$ in a half-open interval of length $N/2$
is at most $\lceil N/4mad\rceil$.  Thus, the count for Type I primes in this case
is at most
\begin{align*}
\sum_{\substack{mad\le 2N+1\\ad~{\rm odd}}}&\left\lceil\frac{N}{4mad}\right\rceil\tau'_2(ma^2d+1)\\
\le&\kappa_2\sum_{\substack{mad\le N/4\\ad~{\rm odd}}}\frac N{4mad}m^{1/6}a^{1/3}d^{1/6}+\kappa_2\sum_{\substack{mad\le 2N+1\\ad~{\rm odd}}}m^{1/6}a^{1/3}d^{1/6}\\
=&S_{2,1}+S_{2,2},
\end{align*}
say, where $\kappa_2=51.3$ from Lemma \ref{lem1}.
We follow the same arguments we made for $S_{1,1},S_{1,2}$, now taking into account that
$a,d$ are odd numbers.  Using Lemma \ref{lem3} with $k=2$, $r=1$, and $\alpha=2/3$,
\[
\sum_{\substack{n\le x\\n~{\rm odd}}}n^{-2/3}<\frac32x^{1/3},
\]
so that
\[
S_{2,1}\le \frac{3\kappa_2N^{4/3}}{2^{11/3}m^{7/6}}(1-2^{-7/6})\zeta(7/6)
<44.3\frac{N^{4/3}}{m^{7/6}}.
\]

For $S_{2,2}$, the analogue of \eqref{eq:twosums} has the two sums with $\tau'_2$ and
with $a,d$ restricted to odd numbers.  Following the
argument with $X$ standing for either $1.01N$ or $2N+1$ and using Lemmas \ref{divisors},
\ref{lem3}, and \eqref{eq:large}, we have
\begin{align*}
\frac12\sum_{\substack{a\le100\\a~{\rm odd}}}\sum_{\substack{d\le X/ma\\d~{\rm odd}}}
&\tau'_{2}(ma^2d+1)\le\frac3{14}\kappa_2m^{1/6}\sum_{\substack{a\le100\\a~{\rm odd}}}
a^{1/3}\Big(\frac{X}{ma}+2\Big)^{7/6}\\
&\le \frac3{14}\kappa_2m^{-1}(1.01X)^{7/6}\sum_{\substack{a\le100\\a~{\rm odd}}}a^{-5/6}
<45.36X^{7/6}m^{-1},
\end{align*}
where we directly computed the $a$-sum and we assumed that $N>m^2$, $m\ge10^9$.
We also have
\begin{align*}
\frac12\sum_{\substack{ad\le X/m\\a,d~{\rm odd}\\a>100}}&\tau'_{2}(ma^2d+1)
\le\frac3{16}\kappa_2m^{1/6}\sum_{\substack{d\le X/100m\\d~{\rm odd}}}d^{1/6}\Big(\frac X{md}+2\Big)^{4/3}\\
&\le\frac3{16}\kappa_2(1-2^{-7/6})\zeta(7/6)(1.02X)^{4/3}/m^{7/6}<36.1X^{4/3}m^{-7/6}.
\end{align*}
Adding these two estimates with the two values of $X$, we have
\[
S_{2,2}<147.8N^{7/6}m^{-1}+127.6N^{4/3}m^{-7/6}.
\]
Thus, assuming $N>m^2$ and $m\ge10^9$,
\[
S_{2,1}+S_{2,2}<\Big(171.9+\frac{147.8}{(N/m)^{1/6}}\Big)N^{4/3}m^{-7/6}<177N^{4/3}m^{-7/6}.
\]

\subsection{The case $\gcd(f,6)=3$.}
In this case we have $madc$ not divisible by 3, and $ma^2d\equiv-1\pmod 3$.  Thus,
we have $d\equiv-m\pmod3$.  The number of integers $c\not\equiv0\pmod 3$ such
that $madc$ is in an interval of length $N/2$ is $\le2\lceil N/6mad\rceil$.  However,
at the top end, namely if $mad>N/2$, then there is at most 1 value of $c$ in play,
and for $mad>1.01N$, we have at most half of $\tau'_3(ma^2d+1)$ as the number of possibilities
for $f$.   Using $2\lceil N/6mad\rceil <2\lfloor N/6mad\rfloor+2$, the number of primes in this case is
at most
\[
S_{3,1}+S_{3,2},
\]
where
\[
S_{3,1}:=\sum_{\substack{mad\le N/6\\d\,\equiv\,-m\kern-5pt\pmod3\\3\,\nmid\, a}}\frac{N}{3mad}\tau'_3(ma^2d+1)
\]
and
 \begin{equation}
 \label{eq:fsum}
S_{3,2}:=f(N/2)+\frac12f(1.01N)+\frac12f(2N+1),
\end{equation}
where 
\begin{equation}
\label{eq:f(x)}
f(x):=\sum_{\substack{mad\le x\\d\,\equiv\,-m\kern-5pt\pmod3\\3\,\nmid\, a}}\tau'_3(ma^2d+1).
\end{equation}

 From Lemma \ref{lem1} with $\kappa_3=32.3$ we have
 \begin{align*}
 S_{3,1}&\le\frac{\kappa_3 N}{3m^{5/6}}\sum_{\substack{ad\le N/6m\\d\,\equiv\,-m\kern-5pt\pmod3\\3\,\nmid\, a}}
 a^{-2/3}d^{-5/6}\\
& <\frac{\kappa_3 N}{3m^{5/6}}\sum_{\substack{d\le N/6m\\d\,\equiv\,-m\kern-5pt\pmod3}}
 d^{-5/6}\cdot2\Big(\frac{N}{6md}\Big)^{1/3},
 \end{align*}
 using Lemma \ref{lem3} with $k=3$ and $r=1,2$.  Summing numerically, we have
 \[
 \sum_{d\,\equiv\,-m\kern-5pt\pmod3}d^{-7/6}
 \le \sum_{d\,\equiv\,1\kern-5pt\pmod 3}d^{-7/6}<2.701,
 \]
so that 
\[
S_{3,1}<32.01N^{4/3}m^{-7/6}.
\]

For $S_{3,2}$ we follow the argument for $S_{2,2}$, getting
\[
\sum_{\substack{a\le100\\3\,\nmid\,a}}\kern-2pt\sum_{\substack{d\le x/ma\\d\,\equiv\,-m\kern-5pt\pmod3}}
\tau'_3(ma^2d+1)\le\frac27\kappa_3\frac{(1.01x)^{7/6}}{m}\sum_{\substack{a\le100\\3\,\nmid\,a}}\frac1{a^{5/6}}<50.1\frac{x^{7/6}}m.
\]
Further,
\[
\sum_{\substack{ad\le x/m\\3\,\nmid\,a\\a>100\\d\,\equiv\,-m\kern-5pt\pmod3}}\tau'_3(ma^2d+1)
\le\frac{\kappa_3}{2}\frac{(1.03x)^{4/3}}{m^{7/6}}
\sum_{\substack{d\,\equiv\,1\kern-5pt\pmod3\\d\le x/m}}\frac1{d^{7/6}}<45.4\frac{x^{4/3}}{m^{7/6}}.
\]
Adding these 2 estimates, we have 
\[
f(x)<\Big(45.4+\frac{50.1}{(x/m)^{1/6}}\Big)\frac{x^{4/3}}{m^{7/6}}<47\frac{x^{4/3}}{m^{7/6}},
\]
where we have been assuming that $x>m^2$ and $m\ge10^9$.  Thus, by \eqref{eq:fsum}
and assuming $N\ge2m^2$, $m\ge10^9$,
\[
S_{3,1}+S_{3,2}<133.7\frac{N^{4/3}}{m^{7/6}}.
\]

\subsection{The case $(f,6)=6$.}
This case is similar to the preceding one, except that we use $\kappa_6=102.7$ in
place of $\kappa_3$, we assume the variables $a,d,c$ are odd (in addition to not
being divisible by 3), and we assume that $d\equiv-m\pmod6$.  The first observation
is that the number of integers $c$ in an interval $(x,x+L]$ coprime to 6 is $\le\lceil(L+1)/3\rceil$.
So, the number of integers $c$ coprime to 6 for which $madc-f$ is in $(N/2,N]$ is $\le\lceil N/6mad+1/3\rceil$.  At the top end, the number of choices for $c$ when $mad\ge N/4$ is $\le 1$,
and the weight of this possible choice is $1/2$ if $mad\ge1.01N$ as before.  Thus, the count is
at most
\begin{equation}
\label{eq:gcd6}
\sum_{\substack{mad\le N/4\\(a,6)=1\\d\,\equiv\,-m\kern-5pt\pmod6}}
\frac N{6mad}\tau'_6(ma^2d+1)
+\frac13g(N/4)+\frac12g(1.01N)+\frac12g(2N+1),
\end{equation}
where
\[
g(x):=
\sum_{\substack{mad\le x\\(a,6)=1\\d\,\equiv\,-m\kern-5pt\pmod6}}
\tau'_6(ma^2d+1).
\]

Using Lemma \ref{lem3},
\begin{align*}
\sum_{\substack{mad\le N/4\\(a,6)=1\\d\,\equiv\,-m\kern-5pt\pmod6}}
&\frac N{6mad}\tau'_6(ma^2d+1)
<\frac{\kappa_6 N}{6m^{5/6}}\sum_{\substack{ad\le N/4m\\(a,6)=1\\d\,\equiv\,-m\kern-5pt\pmod6}}
a^{-2/3}d^{-5/6}\\
&<\frac{\kappa_6 N}{6m^{5/6}}\sum_{\substack{d\le N/4m\\d\,\equiv\,-m\kern-5pt\pmod6}}
\Big(\frac{N}{4md}\Big)^{1/3}d^{-5/6}\\
&<\frac{\kappa_6 N^{4/3}}{6\cdot4^{1/3}m^{7/6}}\sum_{d\,\equiv\,1\kern-5pt\pmod6}d^{-7/6}
<19.23\frac{N^{4/3}}{m^{7/6}}.
\end{align*}

As with the previous cases, we first sum on $a\le100$.  We have
\begin{align*}
\sum_{\substack{a\le100\\(a,6)=1}}\sum_{\substack{d\le x/ma\\d\,\equiv\,-m\kern-5pt\pmod6}}&\tau'_6(ma^2d+1)
\le\kappa_6m^{1/6}\sum_{\substack{a\le100\\(a,6)=1}}\frac17a^{1/3}\Big(\frac{x}{ma}+6\Big)^{7/6}\\
&\le\frac{\kappa_6(1.01x)^{7/6}}{7m}\sum_{\substack{a\le100\\(a,6)=1}}a^{-5/6}
<42.63\frac{x^{7/6}}m.
\end{align*}
Also
\begin{align*}
\sum_{\substack{a>100\\(a,6)=1}}\sum_{\substack{d\le x/ma\\d\,\equiv\,-m\kern-5pt\pmod6}}\tau'_6(ma^2d+1)
&\le \frac{\kappa_6(1.06x)^{4/3}}{4m^{7/6}}\sum_{d\,\equiv\,-m\kern-5pt\pmod6}d^{-7/6}\\
&<49.47\frac{x^{4/3}}{m^{7/6}}.
\end{align*}
Thus, 
\[
g(x)<42.63\frac{x^{7/6}}{m}+49.47\frac{x^{4/3}}{m^{7/6}}.
\]
We conclude that when $m\ge 10^9$ and $N=2m^2$,
\[
\frac13g(N/4)+\frac12g(1.01N)+\frac12g(2N+1)<92.03N^{4/3}m^{-7/6}.
\]
With the previous estimate we have from \eqref{eq:gcd6} that the number of primes $p\le N$ in the case
$\gcd(f,6)=6$ is $\le111.3N^{4/3}/m^{7/6}$.  With the estimates for $S_{j,k}$ for $j=1,2,3$ and $k=1,2$, we have the following result.

\begin{prop}
\label{typeIcount}
When $N=2m^2$ and $m\ge10^9$, the 
number of primes $p\in(N/2, N]$ which have a Type I representation is less than $698N^{4/3}/m^{7/6}$.
\end{prop}

\section{Numerically explicit estimates: primes with Type II representations}
\label{secII}
In this section we  prove the following statement.

\begin{prop}
\label{prop-explicit}
For every natural number $m\ge 6$, there exists a prime 
$
p>m^2,
$
for which \eqref{ESSp} has no Type II solution.
\end{prop}

\begin{remark}
When $m=4$, every prime $p < 10^{13}$ has a Type II solution.
When $m=5$,  every prime $p < 10^{13}$ except $2$ and $5$ has a Type II solution.
We conjecture that these statements also hold for $p>10^{13}$.
\end{remark}

\begin{remark}
The proof actually shows that for all $m\ge 6$, except for $m=7$, there is a prime $p\in (m^2, 2m^2]$ with no Type II solution.
When $m=7$, the first prime $p>7^2$ without a Type II solution is $127$.
\end{remark}

In the proof of Proposition \ref{prop-explicit}, we will make use of the following estimates.
\begin{lemma}\label{lem2}
We have
$$
\sum_{n\le x} \frac{1}{n} \le  1+\log x \qquad (x\ge  1),
$$
$$
\sum_{n\le x} \tau(n)  \le x (1+\log x) \qquad (x \ge 1),
$$
$$
\sum_{n\le x} \frac{\tau(n)}{n}  \le \frac{1}{2} \log^2 x + 2 \log x + 1 :=P(x) \qquad (x \ge 1),
$$
$$
\pi(x)-\pi(x/2)> \frac{x}{2\log x} \qquad (x \ge 3299).
$$

\end{lemma}
\begin{proof}
The first two inequalities are standard exercises.
The third one follows from the second one and partial summation.
The last one follows from Dusart \cite[Eq. (5.5)]{Dus} when $x\ge 10600$, and from direct computation otherwise.
\end{proof}

\begin{proof}[Proof of Proposition \ref{prop-explicit}]
We count the number of primes in the interval $(N/2 ,N]$ covered by Type II solutions.
By Proposition \ref{corII}, if $(x,y,z)$ is a solution to \eqref{ESSp} of Type II, there are natural numbers $a,b,e$ with
$a\le b$ and 
$$
p\equiv - e \mod mab, \qquad e\mid a+b.
$$
By \eqref{type2mabub}, 
$$
mab \le p \frac{m}{m-2} \le N \frac{m}{m-2} = Qm,
$$
where $Q:= N/(m-2)$. 
For given $m,a,b,e$, the number of primes $p$ in $(N/2, N]$ satisfying $p\equiv - e \mod mab$ is 
$$
\le \left\lfloor \frac{N}{2mab}\right\rfloor+1.
$$
The number of primes in $(N/2,N]$ that can be covered by these congruences, by varying the parameters $a, b, e$, is
\begin{equation}\label{Tdef}
\le T:= \sum_{a\le b: ab \le Q } \tau(a+b)\left(\left\lfloor \frac{N}{2mab}\right\rfloor+1\right)=  T_1+T_2,
\end{equation}
say. 
For $T_1$ we may assume $mab\le N/2$, so that
$$
T_1 \le \frac{N}{2m}\sum_{a\le \sqrt{N/2m}} \frac{1}{a} \sum_{a \le b \le N/2am} \frac{\tau(a+b)}{b}.
$$
Writing $n=a+b$, the innermost sum is 
$$
\le \sum_{a \le b \le N/2am} \frac{2\tau(a+b)}{a+b} \le 2\sum_{2a \le n \le N/am} \frac{\tau(n)}{n}
\le 2 P(N/am),
$$
by Lemma \ref{lem2}. We obtain
$$
T_1 \le \frac{N}{m} \sum_{a\le \sqrt{N/2m}} \frac{P(N/am)}{a} \le 
\frac{N}{m}P(N/m)(1+\log \sqrt{N/2m}) .
$$
We let $N=2m^2$. Since $P(x)\le \frac{1}{2}(2+\log x)^2$, we get 
$$
T_1 \le \frac{m}{2}(3+\log m)^3 \qquad (m\ge 4).
$$
Similarly, 
$$
T_2 \le \sum_{a \le \sqrt{Q}} \sum_{b\le Q/a} \tau(a+b) \le \sum_{a \le \sqrt{Q}} \sum_{n\le 2Q/a} \tau(n).
$$
By Lemma \ref{lem2},
$$
T_2 \le \sum_{a \le \sqrt{Q}} \frac{2Q}{a}(1+\log 2Q/a) \le 2Q (1+\log 2Q)(1+\log \sqrt{Q}).
$$
With $N=2m^2$, we obtain
$$
T_2 \le 2m(3+\log m)^2 \qquad (m\ge 18). 
$$
The number of primes in $(N/2,N]=(m^2,2m^2]$ is, 
$$
\pi(N)-\pi(N/2)>\frac{N}{2\log N} \qquad (N\ge  3299),
$$
by Lemma \ref{lem2}.
If $T_1+T_2 $  is less than the number of primes in $(N/2,N]$, then there
are primes in $(N/2,N]$ with no Type II solution.
From the above bounds for $T_1$ and $T_2$, it follows that $T_1+T_2<\frac{N}{2\log N}$ for $m\ge 34000$.
When $16000\le m \le 34000$, we evaluate with a computer the more precise upper bounds
$$
T_1 \le \frac{N}{m} \sum_{a\le \sqrt{N/2m}} \frac{P(n/am)}{a},
\quad
T_2 \le 2Q \sum_{a\le \sqrt{Q}} \frac{1+\log(2Q/a)}{a},
$$
which we also established above, to confirm that $T_1+T_2<\frac{N}{2\log N}$.
When $3000\le m \le 16000$, we evaluate with a computer the original sum $T$ in \eqref{Tdef}
to confirm that $T<\frac{N}{2\log N}$ also holds in this range. 
When $20\le m \le 6000$, a brute force algorithm shows that there is a prime $p$ in $(m^2, 2m^2]$ that has no solution at all to \eqref{ESSp},
and hence no solution of Type II.
Finally, for $6\le m \le 19$, we verify that there is a prime $p>m^2$ that has no Type II solution.
\end{proof}

\section{Proof of Theorem \ref{th-explicit}}

We begin with the following corollary of the work in Section \ref{secII}.
\begin{corollary}
\label{cor}
If $N=2m^2$ and $m\ge 10^9$, then the number of primes in $(N/2,N]$ with a Type II solution is $< \frac{1}{10} N^{4/3} m^{-7/6}$.
\end{corollary}

\begin{proof}
 From the proof of Proposition \ref{prop-explicit}, the number in question is 
$$
\le T_1+T_2 \le \frac{m}{2}(3+\log m)^3+2m(3+\log m)^2.
$$
This is $<\frac{1}{10} N^{4/3} m^{-7/6} = \frac{2^{4/3}}{10}  m^{3/2} $ for $m\ge 10^9$.
\end{proof}

With Proposition \ref{typeIcount} and Corollary \ref{cor} we have that 
when $m\ge10^9$ and $N=2m^2$, the number of primes $p\in(N/2,N]$ for which
$m/p$ is the sum of 3 unit fractions is $< 698.1N^{4/3}/m^{7/6}$.  We contrast this
upper bound with the lower bound in Lemma \ref{lem2} for the number of primes in $(N/2,N]$.
And we find that when $m\ge6.52\times10^9$ and $N=2m^2$, we have
\[
\frac N{2\log N}>698.1\frac{N^{4/3}}{m^{7/6}}.
\]
Hence, Theorem \ref{th-explicit} follows.

\section{Some final thoughts}

In Theorem \ref{th-explicit} we show that the dyadic interval $(m^2,2m^2)$ has a
prime $p$ for which $m/p$ is not the sum of 3 unit fractions, for $m$ beyond a
numerically explicit bound.  One could ask the question for the smaller interval $(m,2m)$.
Here it is a simple exercise to show that $m/(m+1)$ is not the sum of 3 unit fractions
once $m\ge42$.  How about for prime $n$?  Here we have that if $p$ is a prime
in $(m,(6/5)(m-1))$, then $m/p$ is not the sum of 3 unit fractions.  To see this,
note that, as we have seen, if $m/p$ is the sum of 3 unit fractions, then the representation
is either Type I or Type II, so that at least one summand is $\le 1/p$.  The other two
summands have sum $\le 5/6$, so $m/p \le 5/6+1/p$, which implies $p\ge(6/5)(m-1)$,
a contradiction.  Using explicit prime estimates as in \cite{Dus} and a calculation, one can show 
the interval $(m,(6/5)(m-1))$ contains a prime for every $m\ge32$.  In fact, it is not hard
to check smaller $m$'s to see that for each $m\ge14$ there is a prime $p\in(m,2m)$ with
$m/p$ not the sum of 3 unit fractions.

A simple corollary of Theorem \ref{th-mjbound} is that for each $\epsilon>0$ and each 
positive integer $j$, there are infinitely many positive rationals $r<\epsilon$ which are
not the sum of $j$ unit fractions.  Perhaps this statement has a more direct proof?

Talking about $j$ summands, perhaps the natural generalization of the Erd\H os--Straus
conjecture is that for each $m\ge4$ there are at most finitely many $n$ for which
$m/n$ is not the sum of $m-1$ unit fractions.   Might this be provable for some $m$?

\section*{Acknowledgment}
Thanks are due to Paul Pollack for some helpful comments.

\end{document}